\documentclass[12pt, reqno]{amsart}

\usepackage{fullpage}
\usepackage{amssymb}  
\usepackage{rotating}
\usepackage{array}
\usepackage{enumitem}
\usepackage{url}
\usepackage{verbatim}

\theoremstyle{plain}
\newtheorem{theorem}{Theorem}[section]

\newtheorem{lemma}[theorem]{Lemma}

\newtheorem{proposition}[theorem]{Proposition}

\newtheorem{conjecture}[theorem]{Conjecture}
\newtheorem{problem}[theorem]{Problem}

\theoremstyle{remark}

\newtheorem{remark}[theorem]{Remark}

\numberwithin{equation}{section}

\newcommand{\seclabel}[1]{\label{sec:#1}}   
\newcommand{\thmlabel}[1]{\label{thm:#1}}   
\newcommand{\lemlabel}[1]{\label{lem:#1}}   
\newcommand{\prplabel}[1]{\label{prp:#1}}   
\newcommand{\tablabel}[1]{\label{tab:#1}}   
\newcommand{\eqnlabel}[1]{\label{eqn:#1}}   

\newcommand{\secref}[1]{\ref{sec:#1}}   
\newcommand{\thmref}[1]{\ref{thm:#1}}   
\newcommand{\lemref}[1]{\ref{lem:#1}}   
\newcommand{\prpref}[1]{\ref{prp:#1}}   
\newcommand{\eqnref}[1]{\eqref{eqn:#1}} 

\newcommand{\by}[1]{\overset{\eqnref{#1}}=}  

\newcommand{\Aut}{\mathop{\mathrm{Aut}}}			
\newcommand{\Mlt}{\mathop{\mathrm{Mlt}}}         
\newcommand{\Inn}{\mathop{\mathrm{Inn}}}			
\newcommand{\Fix}{\mathrm{Fix}}         

\def\mat#1#2#3#4{\left(\begin{array}{rrrr}#1 & #2\\ #3& #4\end{array}\right)}
\def\trace{\mathrm{tr}}

\newcommand{\setof}[2]{\{#1\,|\,#2\}}   

\newcommand{\ld}{\backslash}					
\newcommand{\inv}{^{-1}}						
\newcommand{\sbl}[1]{\langle#1\rangle}			
\newcommand{\Mace}{\textsc{Mace4}}             
\newcommand{\Prover}{\textsc{Prover9}}         

\title{Nilpotency in automorphic loops of prime power order}

\author{P\v{r}emysl Jedli\v{c}ka$^{\dagger}$}
\author{Michael Kinyon}
\author{Petr Vojt\v{e}chovsk\'y}

\address[Jedli\v{c}ka]{Department of Mathematics \\
Faculty of Engineering \\ Czech University of Life Sciences \\
Kam\'yck\'a 129 \\ 165 21 Prague 6–-Suchdol \\ Czech Republic}
\email[Jedli\v{c}ka]{\url{jedlickap@tf.czu.cz}}

\address[Kinyon and Vojt\v{e}chovsk\'y]{Department of Mathematics \\
University of Denver \\ 2360 S Gaylord St \\ Denver, Colorado 80208 USA}
\email[Kinyon]{\url{mkinyon@math.du.edu}}
\email[Vojt\v{e}chovsk\'y]{\url{petr@math.du.edu}}

\thanks{${}^{\dagger}$ Supported by the Grant Agency of the Czech Republic, grant no.~201/07/P015.}

\begin{document}

\begin{abstract}
A loop is automorphic if its inner mappings are automorphisms. Using so-called associated operations, we show that every commutative automorphic loop of odd prime power order is centrally nilpotent. Starting with suitable elements of an anisotropic plane in the vector space of $2\times 2$ matrices over the field of prime order $p$, we construct a family of automorphic loops of order $p^3$ with trivial center.
\end{abstract}



\maketitle

\section{Introduction}
\seclabel{intro}

A classical result of group theory is that $p$-groups are (centrally) nilpotent. The analogous result does not hold for loops.

The first difficulty is with the concept of a $p$-loop. For a prime $p$, a finite group has order a power of $p$ if and only if each of its elements has order a power of $p$, so $p$-groups can be defined in two equivalent ways. Not so for loops, where the order of an element might not be well defined, and even if it is, the two natural $p$-loop concepts might not be equivalent.

However, there exist several varieties of loops where the analogy with group theory is complete. For instance, a Moufang loop has order a power of $p$ if and only if each of its elements has order a power of $p$, and, moreover, every Moufang $p$-loop is nilpotent \cite{Gl2, GlW}.

We showed in \cite[Thm. 7.1]{JKV2} that a finite commutative automorphic loop has order a power of $p$ if and only if each of its elements has order a power of $p$. The same is true for automorphic loops, by \cite{KiKuPhVo}, \emph{provided} that $p$ is odd; the case $p=2$ remains open.

In this paper we study nilpotency in automorphic loops of prime power order. We prove:

\begin{theorem}\label{Th:Main}
Let $p$ be an odd prime and let $Q$ be a finite commutative automorphic $p$-loop. Then $Q$ is centrally nilpotent.
\end{theorem}

Since there is a (unique) commutative automorphic loop of order $2^3$ with trivial center, cf. \cite{JKV}, Theorem \ref{Th:Main} is best possible in the variety of commutative automorphic loops. (The situation for $p=2$ is indeed complicated in commutative automorphic loops. By
\cite[Prop. 6.1]{JKV}, if a nonassociative finite simple commutative automorphic loop exists, it has exponent two. We now know that no nonassociative finite simple commutative automorphic loop of order less than $2^{12}$ exists \cite{JoKiNaVo}.)

In fact, Theorem \ref{Th:Main} is best possible even in the variety of automorphic loops, because for every prime $p$ we construct here a family of automorphic loops of order $p^3$ with trivial center.

\subsection{Background}
\seclabel{background}

A \emph{loop} $(Q,\cdot)$ is a set $Q$ with a binary operation
$\cdot$ such that (i) for each $x\in Q$, the
\emph{left translation} $L_x : Q\to Q$; $y\mapsto yL_x = xy$
and the \emph{right translation} $R_x : Q\to Q$; $y\mapsto yR_x = yx$
are bijections, and (ii) there exists $1\in Q$ satisfying
$1\cdot x = x\cdot 1 = x$ for all $x\in Q$.

The left and right translations generate the \emph{multiplication group}
$\Mlt Q = \sbl{L_x, R_x \mid x\in Q}$. The \emph{inner mapping group}
$\Inn Q = (\Mlt Q)_1$ is the stabilizer of $1\in Q$.
Standard references for the theory of loops are
\cite{Belousov, Bruck, Pflugfelder}.

A loop $Q$ is \emph{automorphic} (or sometimes just an \emph{A-loop}) if every inner mapping of $Q$ is an
automorphism of $Q$, that is, $\Inn Q \leq \Aut Q$.

The study of automorphic loops was initiated by Bruck and Paige \cite{BP}.
They obtained many basic results, not the least of
which is that automorphic loops are \emph{power-associative}, that is, for all $x$
and all integers $m,n$, $x^m x^n = x^{m+n}$.
In power-associative loops, the \emph{order} of an element
may be defined unambiguously.

For commutative automorphic loops, there now exists a
detailed structure theory \cite{JKV}, as well as constructions
and small order classification results \cite{JKV2}.

Informally, the \emph{center} $Z(Q)$ of a loop $Q$ is the set of all elements
of $Q$ which commute and associate with all other elements. It can be characterized as
$Z(Q) = \Fix(\Inn Q)$, the set of fixed points of the inner mapping group.
(See \S\secref{prelims} for the more traditional definition.)

The center is a \emph{normal} subloop of $Q$, that is, $Z(Q)\varphi = Z(Q)$ for every $\varphi\in\Inn Q$.
Define $Z_0(Q)=\{1\}$, and $Z_{i+1}(Q)$, $i\ge 0$, as the preimage of $Z(Q/Z_i(Q))$
under the canonical projection. This defines the \emph{upper central series}
\[
1 \leq Z_1(Q) \leq Z_2(Q) \leq \cdots \leq Z_n(Q) \leq \cdots \leq Q
\]
of $Q$. If for some $n$ we have $Z_{n-1}(Q) < Z_n(Q) = Q$ then $Q$ is said to be
\emph{(centrally) nilpotent of class} $n$.

\subsection{Summary}

The proof of our main result, Theorem \ref{Th:Main}, is based on a construction from \cite{JKV}. On each commutative automorphic loop $(Q,\cdot)$ which is uniquely $2$-divisible (\emph{i.e.}, the squaring map $x\mapsto x\cdot x$ is a permutation), there
exists a second loop operation $\circ$ such that $(Q,\circ)$ is a Bruck loop (see \S\secref{bruck}), and such that powers of elements in $(Q,\cdot)$ coincide with those in $(Q,\circ)$.

Glauberman \cite{Gl1} showed that for each odd prime $p$ a finite Bruck $p$-loop
is centrally nilpotent. Theorem \ref{Th:Main} will therefore follow immediately from this and from the following result:

\begin{theorem}
\thmlabel{centers}
Let $(Q,\cdot)$ be a uniquely $2$-divisible commutative automorphic loop with associated
Bruck loop $(Q,\circ)$. Then $Z_n(Q,\circ) = Z_n(Q,\cdot)$ for every $n\ge 0$.
\end{theorem}

After reviewing preliminary results in \S\secref{prelims}, we discuss the associated Bruck loop in \S\secref{bruck} and prove Theorem \thmref{centers} in \S\secref{proofs}.

In \S\secref{noncomm}, we use elements of anisotropic planes in the vector space of $2\times 2$ matrices over $GF(p)$ to obtain automorphic loops of order $p^3$ with trivial center. We obtain one such loop for $p=2$ (this turns out to be the unique commutative automorphic loop of order $2^3$ with trivial center), two such loops for $p=3$, three such loops for $p\ge 5$, and at least one (conjecturally, three) such loop for every prime $p\ge 7$.

Finally, we pose open problems in \S\secref{problems}.

\section{Preliminaries}
\seclabel{prelims}

In a loop $(Q,\cdot)$, there are various subsets of interest:
\begin{center}
\begin{tabular}{cllcl}
$\bullet$ & the \emph{left nucleus} &
$N_{\lambda}(Q)$ &$=$& $\setof{a\in Q}{ax\cdot y = a\cdot xy,\ \forall x,y\in Q}$ \\
$\bullet$ & the \emph{middle nucleus} &
$N_{\mu}(Q)$ &$=$ & $\setof{a\in Q}{xa\cdot y = x\cdot ay,\ \forall x,y\in Q}$ \\
$\bullet$ & the \emph{right nucleus} &
$N_{\rho}(Q)$ &$=$& $\setof{a\in Q}{xy\cdot a = x\cdot ya,\ \forall x,y\in Q}$ \\
$\bullet$ & the \emph{nucleus} &
$N(Q)$ &$=$ & $N_{\lambda}(Q)\cap N_{\mu}(Q)\cap N_{\rho}(Q)$ \\
$\bullet$ & the \emph{commutant} &
$C(Q)$ &$=$ & $\setof{a\in Q}{ax = xa,\ \forall x\in Q}$ \\
$\bullet$ & the \emph{center} &
$Z(Q)$ &$=$& $N(Q)\cap C(Q)$\,.
\end{tabular}
\end{center}
\noindent The commutant is not necessarily a subloop, but the nuclei are.

\begin{proposition}\cite{BP}
\prplabel{aut-nuc}
In an automorphic loop $(Q,\cdot)$,  $N_{\lambda}(Q) = N_{\rho}(Q) \le N_{\mu}(Q)$.
In a commutative automorphic loop $(Q,\cdot)$, $Z(Q) = N_{\lambda}(Q)$.
\end{proposition}

We will also need the following (well known) characterization of $C(Q)\cap N_\rho(Q)$:

\begin{lemma}
\lemlabel{nucr-comm}
Let $(Q,\cdot)$ be a loop. Then $a \in C(Q)\cap N_{\rho}(Q)$ if and only if
$L_a L_x = L_x L_a$ for all $x\in Q$.
\end{lemma}

\begin{proof}
If $a \in C(Q)\cap N_{\rho}(Q)$, then for all $x,y\in Q$,
$a\cdot xy = xy\cdot a = x\cdot ya = x\cdot ay$, that is,
$L_a L_x = L_x L_a$. Conversely, if $L_a L_x = L_x L_a$ holds,
then applying both sides to $1$ gives $xa = ax$, \emph{i.e.}, $a\in C(Q)$,
and then $xy\cdot a = a\cdot xy = x\cdot ay = x\cdot ya$, \emph{i.e.},
$a\in N_{\rho}(Q)$.
\end{proof}

The inner mapping group $\Inn Q$ of a loop $Q$ has a standard set of generators
\begin{displaymath}
    L_{x,y} = L_x L_y L_{yx}\inv,\quad
    R_{x,y} = R_x R_y R_{xy}\inv,\quad
    T_x  = R_xL_x\inv,
\end{displaymath}
for $x$, $y\in Q$. The property of being an automorphic loop can therefore be expressed equationally by demanding that the permutations $L_{x,y}$, $R_{x,y}$, $T_x$ are homomorphisms. In particular, if $Q$ is a commutative loop then $Q$ is automorphic if and only if
\begin{displaymath}
    (uv)L_{x,y} = uL_{x,y}\cdot vL_{x,y}
\end{displaymath}
for every $x$, $y$, $u$, $v$.

We can conclude that (commutative) automorphic loops form a variety in the sense of universal algebra, and are therefore closed under subloops, products, and homomorphic images.

We will generally compute with translations whenever possible, but it will sometimes be convenient to work directly with the loop operations. Besides the loop multiplication, we also have the \emph{left division} operation $\ld : Q\times Q\to Q$ which satisfies
\[
x\ld (xy) = x (x\ld y) = y\,.
\]
The \emph{division permutations} $D_x : Q\to Q$ defined by
$y D_x = y\ld x$  are also quite useful, as is the \emph{inversion permutation} $J : Q\to Q$ defined by $xJ = x D_1 = x\inv$ in any power-associative loop.

If $Q$ is a commutative automorphic loop then for all $x$, $y\in Q$ we have
\begin{align}
    x L_{y,x} &= x,               \eqnlabel{lem1} \\
    L_{y,x} L_{x\inv} &= L_{x\inv} L_{y,x}, \eqnlabel{lem1.5} \\
    y L_{y,x} &= ((xy)\ld x)\inv, \eqnlabel{lem2}\\
    L_{x\inv,y\inv} &= L_{x,y}, \eqnlabel{laip}\\
    D_{x^2} &= D_x J D_x, \eqnlabel{Ds}
\end{align}
where the first two equalities follow from \cite[Lem. 2.3]{JKV}, \eqnref{lem2} from \cite[Lem. 2.5]{JKV}, \eqnref{laip} is an immediate consequence of \cite[Lem. 2.7]{JKV}, and \eqnref{Ds} is \cite[Lem. 2.8]{JKV}. In addition, commutative automorphic loops satisfy the \emph{automorphic inverse property}
\begin{equation}\eqnlabel{aip}
(xy)\inv = x\inv y\inv \quad\text{and}\quad (x\ld y)\inv = x\inv \ld y\inv,
\end{equation}
by \cite[Lem. 2.6]{JKV}.

Finally, as in \cite{JKV}, in a commutative automorphic loop $(Q,\cdot)$, it will be
convenient to introduce the permutations
\[
P_x = L_x L_{x\inv}\inv = L_{x\inv}\inv L_x,
\]
where the second equality follows from \cite[Lem. 2.3]{JKV}.

\begin{lemma}
\lemlabel{lem3}
For all $x,y$ in a commutative automorphic loop $(Q,\cdot)$
\begin{align}
(x\inv) P_{xy} &= xy^2, \eqnlabel{3.2} \\
x\cdot xP_y &= (xy)^2. \eqnlabel{lem3}
\end{align}
\end{lemma}

\begin{proof}
Equation \eqnref{3.2} is from \cite[Lem 3.2]{JKV}. Replacing $x$ with $x\inv$ and $y$ with $xy$ in \eqnref{3.2} yields $x P_{x\inv \cdot xy} = x\inv (xy)^2$ and $x P_{x\inv \cdot xy} = x L_{x,x\inv} P_{x\inv \cdot xy} = x L_{x,x\inv} P_{y L_{x,x\inv}}$. Now, for every automorphism $\varphi$ of $Q$ we have $x\varphi P_{y\varphi} = (y\varphi)\inv \ld (y\varphi x\varphi) = (y\inv \ld (yx))\varphi = xP_y\varphi$. Thus  $x\inv (xy)^2  = x L_{x,x\inv} P_{y L_{x,x\inv}} = x P_y L_{x,x\inv}$. Canceling $x\inv$ on both sides, we obtain \eqnref{lem3}.
\end{proof}

\section{The associated Bruck loop}
\seclabel{bruck}

A loop $(Q,\circ)$ is said to be a (left) \emph{Bol loop} if it satisfies
the identity
\begin{equation}
\eqnlabel{bol}
(x\circ (y\circ x))\circ z = x\circ (y\circ (x\circ z)) \,.
\end{equation}
A Bol loop is a \emph{Bruck loop} if it also satisfies the
automorphic inverse property $(x\circ y)\inv = x\inv \circ y\inv$.
(Bruck loops are also known as \emph{$K$-loops} or \emph{gyrocommutative gyrogroups}.)

The following construction is the reason for considering Bruck loops in this paper.
Let $(Q,\cdot)$ be a uniquely $2$-divisible commutative automorphic loop.
Define a new operation $\circ$ on $Q$ by
\[
x \circ y = \lbrack x\inv \ld (x y^2)\rbrack^{1/2} = \lbrack (y^2)P_x \rbrack^{1/2}\,.
\]
By \cite[Lem. 3.5]{JKV}, $(Q,\circ)$ is a Bruck loop, and powers in $(Q,\circ)$ coincide with powers in $(Q,\cdot)$.

Since we will work with translations in both $(Q,\cdot)$ and $(Q,\circ)$, we will denote left translations in $(Q,\circ)$ by $L_x^{\circ}$. For instance, we can express the fact that every Bol loop $(Q,\circ)$ is \emph{left power alternative} by
\begin{equation}
\eqnlabel{lalt}
(L_x^{\circ})^n = L_{x^n}^{\circ}
\end{equation}
for all integers $n$.

\begin{proposition}\cite[Thm. 5.10]{kiechle}
\prplabel{bol-nuclei}
Let $(Q,\circ)$ be a Bol loop. Then $N_{\lambda}(Q,\circ) = N_{\mu}(Q,\circ)$.
If, in addition, $(Q,\circ)$ is a Bruck loop, then
$N_{\lambda}(Q,\circ) = Z(Q,\circ)$.
\end{proposition}

In the uniquely $2$-divisible case, we can say more about the center.

\begin{lemma}
\lemlabel{bruck-center}
Let $(Q,\circ)$ be a uniquely $2$-divisible Bol loop. Then
$Z(Q,\circ) = C(Q,\circ) \cap N_{\rho}(Q,\circ)$.
\end{lemma}

\begin{proof}
One inclusion is obvious. For the other, suppose
$a\in C(Q,\circ) \cap N_{\rho}(Q,\circ)$. Then for all $x,y\in Q$,
\begin{alignat*}{2}
(x^2 \circ a)\circ y &\by{lalt} (x\circ (x\circ a))\circ y
&&= (x\circ (a\circ x))\circ y \\
&\by{bol} x\circ (a\circ (x\circ y))
&&= x\circ (x\circ (a\circ y)) \\
&\by{lalt} x^2 \circ (a\circ y)\,, &&
\end{alignat*}
where we used $a\in C(Q,\circ)$ in the second equality and
Lemma \lemref{nucr-comm} in the fourth. Since squaring is a permutation,
we may replace $x^2$ with $x$ to get $(x\circ a)\circ y = x\circ (a\circ y)$ for all
$x,y\in Q$. Thus $a\in N_{\mu}(Q,\circ) = N_{\lambda}(Q,\circ)$ (Proposition \prpref{bol-nuclei}),
and so $a\in Z(Q,\circ)$.
\end{proof}

\begin{lemma}
\lemlabel{center-char}
Let $(Q,\cdot)$ be a uniquely $2$-divisible commutative automorphic loop with
associated Bruck loop $(Q,\circ)$. Then
$a\in Z(Q,\circ)$ if and only if, for all $x\in Q$,
\begin{equation}
\eqnlabel{center-char}
P_a P_x = P_x P_a\,.
\end{equation}
\end{lemma}

\begin{proof}
By Lemmas \lemref{nucr-comm} and \lemref{bruck-center}, $a\in Z(Q,\circ)$ if and only if
the identity $a\circ (x\circ y) = x\circ (a\circ y)$ holds for all $x,y\in Q$. This
can be written as $\lbrack (y^2) P_x P_a \rbrack^{1/2} = \lbrack (y^2) P_a P_x \rbrack^{1/2}$.
Squaring both sides and using unique $2$-divisibility to replace $y^2$ with $y$, we have
$(y) P_x P_a = (y) P_a P_x$ for all $x,y\in Q$.
\end{proof}

\section{Proofs of the Main Results}
\seclabel{proofs}

Throughout this section, let $(Q,\cdot)$ be a uniquely $2$-divisible, commutative automorphic
loop with associated Bruck loop $(Q,\circ)$.

\begin{lemma}
\lemlabel{483}
If $a\in Z(Q,\circ)$, then for all $x\in Q$,
\begin{equation}
\eqnlabel{483}
x L_{a\ld x,a} = x L_{a\ld x\inv,a}\,.
\end{equation}
\end{lemma}

\begin{proof}
First,
\begin{alignat*}{2}
x^{-2} &= x^{-2} L_{a\inv}\inv L_{a\inv}
&&= a\inv D_{x^{-2}} L_{a\inv} \\
&\by{aip} a D_{x^2} J L_{a\inv}
&&\by{Ds} a D_x J D_x J L_{a\inv} \\
&\by{aip} a D_x D_{x\inv} L_{a\inv}
&&= (x\inv) L_{a\ld x}\inv L_{a\inv}\,.
\end{alignat*}
Thus we compute
\begin{alignat}{2}
(x^{-2}) L_{a\ld x,a}
&= (x\inv) L_{a\ld x}\inv L_{a\inv} L_{a\ld x,a}
&&\by{lem1.5} (x\inv) L_{a\ld x}\inv L_{a\ld x,a} L_{a\inv} \notag\\
&= (x\inv) L_a L_x\inv L_{a\inv}
&&= a L_{x\inv} L_x\inv L_{a\inv} \eqnlabel{483tmp}\\
&= a P_{x\inv} L_{a\inv}\,. &&\notag
\end{alignat}
Since $a\inv \in Z(Q,\circ)$, we may also apply \eqnref{483tmp}
with $a\inv$ in place of $a$, and will do so in the next calculation. Now
\begin{alignat*}{2}
a P_{x\inv} L_{a\inv} &= a P_{x\inv} P_{a\inv} L_a
&&\by{center-char} a P_{a\inv} P_{x\inv} L_a \\
&= a\inv P_{x\inv} L_a
&&\by{483tmp} (x^{-2}) L_{a\inv \ld x,a\inv} \\
&\by{aip} (x^{-2}) L_{(a\ld x\inv)\inv,a\inv}
&&\by{laip} (x^{-2}) L_{a\ld x\inv,a}\,,
\end{alignat*}
where we used $a\inv \in Z(Q,\circ)$ in the second equality.

Putting this together with \eqnref{483tmp}, we have
$(x^{-2}) L_{a\ld x,a} = (x^{-2}) L_{a\ld x\inv,a}$ for all $x\in Q$.
Since inner mappings are automorphisms, this implies
$(x L_{a\ld x,a})^{-2} = (x L_{a\ld x\inv,a})^{-2}$. Taking
inverses and square roots, we have the desired result.
\end{proof}

\begin{lemma}
\lemlabel{489}
If $a\in Z(Q,\circ)$, then for all $x\in Q$,
\begin{align}
    (a\ld x) L_{a\ld x\inv,a} &= (x\ld a)\inv,\eqnlabel{489}\\
    x\inv \cdot xP_a &= a^2.\eqnlabel{ugh}
\end{align}
\end{lemma}

\begin{proof}
We compute
\[
(a\ld x) L_{a\ld x\inv,a} = a\ld (x L_{a\ld x\inv,a})
\by{483} a \ld (x L_{a\ld x,a})
\by{lem1} (a\ld x) L_{a\ld x,a}
\by{lem2} (x\ld a)\inv \,,
\]
where we used $L_{a\ld x\inv,a}\in \Aut Q$ in the first equality and $L_{a\ld x,a}\in \Aut Q$ in the third equality.

To show \eqnref{ugh}, we compute
\begin{alignat*}{2}
x\inv \cdot xP_a &= (x\inv)L_{a\inv \ld (ax)}
&&= (x\inv)L_{a\inv \ld (ax)} L_{a\inv} L_{ax}\inv L_{ax} L_{a\inv}\inv \\
&= (a\ld (ax))\inv L_{a\inv \ld (ax),a\inv} L_{ax} L_{a\inv}\inv
&&\by{aip}  (a\inv\ld (ax)\inv) L_{a\inv \ld (ax),a\inv} L_{ax} L_{a\inv}\inv \\
&\by{489} ( (ax)\inv \ld a\inv)\inv L_{ax} L_{a\inv}\inv
&&\by{aip} ((ax) \ld a ) L_{ax} L_{a\inv}\inv \\
&= a L_{a\inv}\inv &&= a^2\,.
\end{alignat*}
Note that in the fifth equality, we are applying \eqnref{489} with
$a\inv$ in place of $a$ and $(ax)\inv$ in place of $x$.
\end{proof}

\begin{lemma}
\lemlabel{leftmult}
If $a\in Z(Q,\circ)$, then $L_a = L_a^{\circ}$, and for all integers $n$
\begin{equation}
\eqnlabel{poweralt}
L_a^n = L_{a^n}.
\end{equation}
\end{lemma}

\begin{proof}
For $x\in Q$, we compute
\[
(a\circ x)^2 = (x\circ a)^2 = (a^2) P_x
\by{ugh} x P_a L_{x\inv} P_x = x\cdot xP_a
\by{lem3} (ax)^2\,.
\]
Taking square roots, we have $a\circ x = ax$, as desired. Then
$L_a^n = (L_a^{\circ})^n \by{lalt} L_{a^n}^{\circ} = L_{a^n}$.
\end{proof}

\begin{lemma}
\lemlabel{Ps}
If $a\in Z(Q,\circ)$, then for all $x\in Q$,
\begin{equation}
\eqnlabel{Ps}
P_{xa} = P_x P_a \,.
\end{equation}
\end{lemma}

\begin{proof}
For each $y\in Q$,
\[
y P_{xa} = y P_{ax} = \lbrack ax\circ y^{1/2} \rbrack^2
= \lbrack (a\circ x)\circ y^{1/2} \rbrack^2
= \lbrack a\circ (x\circ y^{1/2}) \rbrack^2
= y P_x P_a\,,
\]
using Lemma \lemref{leftmult} in the third equality and $a\in Z(Q,\circ)$ in the
fourth.
\end{proof}

\begin{lemma}
\lemlabel{close}
If $a\in Z(Q,\circ)$, then $a^2 \in Z(Q,\cdot)$.
\end{lemma}

\begin{proof}
We compute
\begin{alignat*}{2}
L_{a^2} L_x &\by{poweralt} L_a^2 L_x
&&= L_a L_{a,x} L_{xa} \\
&\by{laip} L_a L_{a\inv,x\inv} L_{xa}
&&= L_a L_{a\inv} L_{x\inv} L_{x\inv a\inv}\inv L_{xa} \\
&\by{poweralt}  L_{x\inv} L_{x\inv a\inv}\inv L_{xa}
&&\by{aip} L_{x\inv} L_{(xa)\inv}\inv L_{xa} \\
&= L_{x\inv} P_{xa}
&&\by{Ps} L_{x\inv} P_x P_a \\
&= L_x L_a L_{a\inv}\inv
&&\by{poweralt} L_x L_a^2 \\
&\by{poweralt} L_x L_{a^2}\,. &&
\end{alignat*}
By Lemma \lemref{nucr-comm}, it follows that $a^2 \in N_{\rho}(Q,\cdot)$, and $N_\rho(Q,\cdot) = Z(Q,\cdot)$ by Proposition \prpref{aut-nuc}.
\end{proof}

\begin{lemma}
\lemlabel{inclusion1}
Let $(Q,\cdot)$ be a uniquely $2$-divisible commutative automorphic loop with
associated Bruck loop $(Q,\circ)$. Then $Z(Q,\circ)\subseteq Z(Q,\cdot)$.
\end{lemma}

\begin{proof}
Assume that $a\in Z(Q,\circ)$. Then $a^2\in Z(Q,\cdot)$ by Lemma \lemref{close}, and thus $(aL_{x,y})^2 =a^2 L_{x,y} = a^2$ for every $x$, $y\in Q$. Taking square roots yields $aL_{x,y}=a$, that is, $a\in Z(Q,\cdot)$.
\end{proof}

Now we prove Theorem \thmref{centers}, that is, we show that the upper central series
of $(Q,\cdot)$ and $(Q,\circ)$ coincide.

\begin{proof}[Proof of Theorem \thmref{centers}]
Since each $Z_n(Q)$ is the preimage of $Z(Q/Z_{n-1}(Q))$ under the canonical
projection, it follows by induction that it suffices to show $Z(Q,\circ) = Z(Q,\cdot)$.
One inclusion is Lemma \lemref{inclusion1}. For the other, suppose $a\in Z(Q,\cdot)$.
Then $P_a P_x = L_a L_{a\inv}\inv L_x L_{x\inv}\inv
= L_x L_{x\inv}\inv L_a L_{a\inv}\inv = P_x P_a$, and so
$a\in Z(Q,\circ)$ by Lemma \lemref{center-char}.
\end{proof}

\begin{proof}[Proof of Theorem \ref{Th:Main}]
For an odd prime $p$, let $Q$ be a commutative automorphic $p$-loop with
associated Bruck loop $(Q,\circ)$. By \cite{Gl1}, $(Q,\circ)$ is centrally
nilpotent of class, say, $n$. By Theorem \thmref{centers}, $Q$ is also centrally
nilpotent of class $n$.
\end{proof}

\section{From anisotropic planes to automorphic $p$-loops with trivial nucleus}
\seclabel{noncomm}

We proved in \cite{JKV2} that a commutative automorphic loop of order $p$, $2p$, $4p$, $p^2$, $2p^2$ or $4p^2$ is an abelian group. For every prime $p$ there exist nonassociative commutative automorphic loops of order $p^3$. These loops have been classified up to isomorphism in \cite{deBaGrVo}, where the announced Theorem \ref{Th:Main} has been used to guarantee nilpotency for $p$ odd.

Without commutativity, we do not even know whether automorphic loops of order $p^2$ are associative! Nevertheless we show here that the situation is much more complicated than in the commutative case already for loops of order $p^3$. Namely, we construct a family of automorphic loops of order $p^3$ with trivial nucleus.

\subsection{Anisotropic planes}

Let $F$ be a field, $V$ a finite-dimensional vector space over $F$, and $q:V\to F$ a quadratic form. A subspace $W\le V$ is \emph{isotropic} if $q(x) = 0$ for some $0\ne x\in W$, else it is \emph{anisotropic}.

It is well known that if $F$ is a finite field and $\dim{V}\ge 3$ then $V$ must be isotropic. (See \cite[Thm. 3.8]{Sch} for a proof in odd characteristic.) Moreover, if $F = GF(p)$ then there is a unique anisotropic space of dimension $2$ over $F$ up to isometry. (See \cite[Cor. 3.10]{Sch} for $p$ odd. If $p=2$ and $V=\langle x,y\rangle$, we must have $q(0)=0$, $q(x)=q(y)=q(x+y)=1$ for $V$ to be anisotropic.) Let us call anisotropic subspaces of dimension two \emph{anisotropic planes}.

Since our construction is based on elements of anisotropic planes rather than on the planes themselves, we will first have a detailed look at anisotropic planes in $M(2,F)$, the vector space of $2\times 2$ matrices over $F$. The determinant
\begin{displaymath}
    \det:M(2,F)\to F,\quad \det\mat{a_1}{a_2}{a_3}{a_4} = a_1a_4-a_2a_3
\end{displaymath}
is a quadratic form on $M(2,F)$. If $FC\oplus FD$ is an anisotropic plane in $M(2,F)$ then $C^{-1}(FC\oplus FD)$ is also anisotropic, and hence, while looking for anisotropic planes, it suffices to consider subspaces $FI\oplus FA$, where $I$ is the identity matrix and $A\in GL(2,F)$.

\begin{lemma}\label{Lm:AnisotropicPlanes}
With $A\in M(2,F)$, the subspace $FI\oplus FA$ is an anisotropic plane if and only if the characteristic polynomial $\det(A-\lambda I) = \lambda^2 - \trace(A)\lambda + \det(A)$ has no roots in $F$.
\end{lemma}
\begin{proof}
The subspace $FI\oplus FA$ is anisotropic if and only if $\det(\lambda I + \mu A)\ne 0$ for every $\lambda$, $\mu$ such that $(\lambda,\mu)\ne (0,0)$, or, equivalently, if and only if $\det(A-\lambda I)\ne 0$ for every $\lambda$. We have $\det(A-\lambda I) = \lambda^2 - \trace(A)\lambda + \det(A)$.
\end{proof}

We will now impose additional conditions on anisotropic planes over finite fields and establish their existence or non-existence. We will take advantage of the following strong result of Perron \cite[Thms. 1 and 3]{Perron} concerning additive properties of the set of quadratic residues.

A nonzero element $a\in GF(p)$ is a \emph{quadratic residue} if $a=b^2$ for some $b\in GF(p)$. A nonzero element $a\in GF(p)$ that is not a quadratic residue is a \emph{quadratic nonresidue}.

\begin{theorem}[Perron]\label{Th:Perron}
Let $p$ be a prime, $N_p$ the set of quadratic nonresidues, and $R_p = \{a\in GF(p);\;a$ is a quadratic residue or $a=0\}$.
\begin{enumerate}
\item[(i)] If $p=4k-1$ and $a\ne 0$ then $|(R_p+a)\cap R_p| = k = |(R_p+a)\cap N_p|$.
\item[(ii)] If $p=4k+1$ and $a\ne 0$ then $|(R_p+a)\cap R_p| = k+1$, $|(R_p+a)\cap N_p| = k$.
\end{enumerate}
\end{theorem}

\begin{lemma}\label{Lm:Additive}
Let $p\ge 5$ be a prime. Then there is a quadratic nonresidue $a$ and quadratic residues $b$, $c$ such that $b-a$ is a quadratic residue and $c-a$ is a quadratic nonresidue.
\end{lemma}
\begin{proof}
Let $p=4k\pm 1$. If $k\ge 3$ then we are done by Theorem \ref{Th:Perron}, since $|(R_p-a)\cap R_p|$, $|(R_p-a)\cap N_p|\ge 3$. (We need $k\ge 3$ to be able to pick $b\in R_p\setminus\{0\}$ such that $b-a\in R_p\setminus\{0\}$.) If $p=7$ then $a=3$, $b=4$, $c=1$ do the job. If $p=5$ then $a=2$, $b=1$, $c=4$ do the job.
\end{proof}

\begin{lemma}\label{Lm:3Types}
Let $p$ be a prime and $F=GF(p)$.
\begin{enumerate}
\item[(i)] There is $A\in GL(2,p)$ such that $\trace(A) = 0$ and $FI\oplus FA$ is anisotropic if and only if $p\ne 2$.
\item[(ii)] There is $A\in GL(2,p)$ such that $\trace(A)\ne 0$, $\det(A)$ is a quadratic residue modulo $p$ and $FI\oplus FA$ is anisotropic if and only if $p\ne 3$.
\item[(iii)] There is $A\in GL(2,p)$ such that $\trace(A)\ne 0$, $\det(A)$ is a quadratic nonresidue modulo $p$ and $FI\oplus FA$ is anisotropic if and only if $p\ne 2$.
\end{enumerate}
\end{lemma}
\begin{proof}
Let $p\ge 3$. For a quadratic nonresidue $a$ and any $b\in F$, let
\begin{displaymath}
    M_{a,b} = \mat{-b}{1}{a}{-b}.
\end{displaymath}
Since $M_{a,b} = M_{a,0}-bI$, we have $FI\oplus FM_{a,b} = FI\oplus FM_{a,0}$. Now, $\trace(M_{a,0})=0$, $\det(M_{a,0}-\lambda I) =\lambda^2-a$ has no roots, so $FI\oplus FM_{a,b}$ is anisotropic by Lemma \ref{Lm:AnisotropicPlanes}. Moreover, if $b\ne 0$ then $\trace(M_{a,b}) = -2b\ne 0$ and $\det(M_{a,b}) = b^2-a$.

If $p\ge 5$, Lemma \ref{Lm:Additive} implies that the parameters $a$ and $b\ne 0$ can be chosen so that $\det(M_{a,b})$ is a quadratic residue or nonresidue as we please.

Let $p=3$. Then $\det(M_{2,2})$ is a quadratic nonresidue. If $\trace(A)\ne 0$ and $\det(A)$ is a quadratic residue then $\det(A)=1$ and $\det(A-\lambda I)$ is equal to either $\lambda^2+\lambda+1$ (with root $1$) or $\lambda^2-\lambda+1$ (with root $-1$), so $FI\oplus FA$ is isotropic.

Let $p=2$. Then
\begin{displaymath}
    \mat{0}{1}{1}{1}
\end{displaymath}
satisfies the conditions of (ii). The only elements $A\in GL(2,p)$ with $\trace(A)=0$ are
\begin{displaymath}
    \mat{0}{1}{1}{0},\quad\mat{1}{0}{1}{1},\quad\mat{1}{1}{0}{1},
\end{displaymath}
all with $\det(A+I)=0$, so $FI\oplus FA$ is isotropic. There is no matrix satisfying the conditions of (iii) because there are no quadratic nonresidues in $GF(2)$.
\end{proof}

Let $p$ be a prime and $F=GF(p)$. Call an element $A\in GL(2,p)$ of an anisotropic plane $FI\oplus FA$ of \emph{type $1$} if $\trace(A)=0$, of \emph{type $2$} if $\trace(A)\ne 0$ and $\det(A)$ is a quadratic residue, and of \emph{type 3} if $\trace(A)\ne 0$ and $\det(A)$ is a quadratic nonresidue.

Note that for a fixed prime $p$ we can find elements $A$ of all possible types (with the restrictions of Lemma \ref{Lm:3Types}) in a single anisotropic plane. This is because we only used matrices $A=M_{a,b}$ with the same $a$ in the proof of Lemma \ref{Lm:3Types}, and $FI\oplus FM_{a,0} = FI\oplus FM_{a,b}$.

\subsection{Automorphic loops of order $p^3$ with trivial nucleus}

Let $A\in GL(2,p)$ be such that $FI\oplus FA$ is an anisotropic plane.
Define a binary operation on $F\times (F\times F)$ by
\begin{equation}\label{Eq:Construction}
    (a,x)\cdot (b,y) = (a + b, x(I+bA) + y(I-aA))
\end{equation}
and call the resulting groupoid $Q(A)$. Since
\begin{displaymath}
   U_a = I+aA
\end{displaymath}
is invertible for every $a\in F$, we see that $Q(A)$ is a loop (see Remark \ref{Rm:Det}),
and in fact, straightforward calculation shows that
\begin{align*}
    (b,y)L^{-1}_{(a,x)} &= (b-a, (y-xU_{b-a})U_{-a}^{-1})\,,\\
    (b,y)R^{-1}_{(a,x)} &= (b-a, (y-xU_{a-b})U_a^{-1})\,.
\end{align*}

\begin{lemma}
\lemlabel{automorphisms}
Let $F=GF(p)$. Let $A\in GL(2,p)$ be such that $FI\oplus FA$ is an anisotropic plane in $M(2,p)$.
For each $z\in F\times F$ and each $C\in GL(2,p)$ satisfying $CA = AC$, define
$\varphi_{z,C} : F\times (F\times F) \to F\times (F\times F)$ by
\[
(a,x)\varphi_{z,C} = (a,az + xC)\,.
\]
Then $\varphi_{z,C}$ is an automorphism of $Q(A)$.
\end{lemma}

\begin{proof}
We compute
\begin{align*}
(a,x)\varphi_{z,C}\cdot (b,y)\varphi_{z,C} &=
(a,az+xC)\cdot (b,bz+ yC) \\
&= (a+b, (az+xC)U_b+ (bz+yC)U_{-a} ) \\
&= (a+b, (a+b)z + xCU_b + yCU_{-a} + abzA - abzA) \\
&= (a+b, (a+b)z + (xU_b + yU_{-a})C ) \\
&= [(a,x)\cdot (b,y)]\varphi_{z,C}\,,
\end{align*}
where we have used $CA = AC$ in the fourth equality. Since
$\varphi_{z,C}$ is clearly a bijection, we have the desired result.
\end{proof}

\begin{proposition}\label{Pr:Constr}
Let $F=GF(p)$. Let $A\in GL(2,p)$ be such that $FI\oplus FA$ is an anisotropic plane in $M(2,p)$.
Then the loop $Q = Q(A)$ defined on $F\times (F\times F)$ by \eqref{Eq:Construction} is an automorphic loop of order $p^3$ and exponent $p$
with $N_{\mu}(Q) = \setof{(0,x)}{x\in F\times F} \cong F\times F$ and
$N_{\lambda}(Q) = N_{\rho}(Q) = 1$. In particular, $N(Q) = Z(Q) = 1$ and so $Q$ is
not centrally nilpotent. In addition, if $p=2$ then $C(Q) = Q$, while if $p>2$,
then $C(Q) = 1$.
\end{proposition}

\begin{proof}
Easy calculations show that
the standard generators of the inner mapping group of $Q(A)$ are
\begin{align}
    (b,y)T_{(a,x)} &= (b,(x(U_{-b} - U_b) + yU_a)U_{-a}^{-1})\,,\notag\\
    (c,z)R_{(a,x),(b,y)} &= (c,(zU_a U_b + y(U_{-c-a} - U_{-c} U_{-a})) U_{a+b}^{-1})\,,\label{Eq:InnerMaps}\\
    (c,z)L_{(a,x),(b,y)} &= (c,(zU_{-a} U_{-b} + y(U_{c+a} - U_c U_a)) U_{-a-b}^{-1})\,.\notag
\end{align}
Since $U_{-b} - U_b = -2bA$ and
$U_{c+a} - U_c U_a = U_{-c-a} - U_{-c} U_{-a} = -caA^2$, we find that each of these generators is of the form $\varphi_{u,C}$ for an appropriate $u\in F\times F$ and $C\in GL(2,p)$ commuting with $A$. Specifically, we have
\begin{alignat*}{3}
    T_{(a,x)} &= \varphi_{u,C} && \quad\text{where}\quad u = -2xAU_{-a}\inv &&\quad\text{and}\quad C = U_aU_{-a}\inv \,, \\
    R_{(a,x),(b,y)} &= \varphi_{u,C} && \quad\text{where}\quad u = -ayA^2 U_{a+b}\inv &&\quad\text{and}\quad C = U_a U_b U_{a+b}\inv\,,\\
    L_{(a,x),(b,y)} &= \varphi_{u,C} && \quad\text{where}\quad u = -ayA^2 U_{-a-b}\inv
    &&\quad\text{and}\quad C = U_{-a} U_{-b} U_{-a-b}\inv \,.
\end{alignat*}
Hence $Q(A)$ is automorphic by Lemma \lemref{automorphisms}.

An easy induction shows that powers in $Q(A)$ and in $F\times (F\times F)$ coincide, so $Q(A)$ has exponent $p$.

Suppose that $(a,x)\in N_{\mu}(Q)$. Then $(c,z)R_{(a,x),(b,y)}=(c,z)$ for every $(c,z), (b,y)$. Thus $(zU_aU_b + y(U_{-c-a}-U_{-c}U_{-a}))U_{a+b}^{-1} =z$ for every $(c,z), (b,y)$.
With $z = 0$, we have $y(U_{-c-a}-U_{-c}U_{-a}) = -cayA^2 =0$ for every $y$, hence
$caA^2 = 0$ for every $c$, and $a = 0$ follows. On the other hand, clearly
$(0,x)\in N_{\mu}(Q)$ for every $x$. We have thus shown
$N_{\mu}(Q) = \setof{(0,x)}{x\in F\times F} \cong F\times F$.

Suppose that $(c,z)\in N_{\lambda}(Q)$. By Proposition \prpref{aut-nuc}, $N_{\lambda}(Q) = N_{\rho}(Q) \le N_{\mu}(Q)$, so $c=0$. We then must have $(0,z)R_{(a,x),(b,y)} = (0,z)$, or $zU_aU_bU_{a+b}^{-1}=z$, or $abzA^2=0$ for every $a$, $b$. In particular, $zA^2=0$ and $z=0$. We have proved $N_{\lambda}(Q)=1$.

If $p=2$, then since $U_a=U_{-a}$, it follows that $Q$ is commutative.
Now assume that $p>2$ and let $(a,x)\in C(Q)$. Then $x(U_b-U_{-b}) = y(U_a-U_{-a})$,
that is, $2bxA = 2ayA$ for every $(b,y)\in Q$. With $b=0$ we deduce that $2ayA=0$ for every $y$, thus
$0 = 2aA$, or $a=0$. Then $2bxA=0$, and with $b=1$ we deduce $2xA=0$, or $x=0$.
We have proved that $C(Q) = 1$.
\end{proof}

\begin{remark}
The construction $Q(A)$ works for every real anisotropic plane $\mathbb RI\oplus \mathbb RA$ and results in an automorphic loop on $\mathbb R^3$ with trivial center. We believe that this is the first time a smooth nonassociative automorphic loop has been constructed.
\end{remark}

\begin{remark}\label{Rm:Det}
The groupoid $Q(A)$ is an automorphic loop as long as $I+aA$ is invertible for every $a\in F$, which is a weaker condition than having $FI\oplus FA$ an anisotropic plane, as witnessed by $A=0$, for instance. But we claim that nothing of interest is obtained in the more general case:

Let us assume that $A\in M(2,F)$ is such that $I+aA$ is invertible for every $a\in F$ but $FI\oplus FA$ is not anisotropic. Then $\det(A)=0$ and $\det(A-\lambda I) = \lambda^2-\trace(A)\lambda = \lambda(\lambda-\trace(A))$ has no nonzero solutions. Hence $\trace(A)=0$ and $A^2=0$. The loop $Q=Q(A)$ is still an automorphic loop by the argument given in the proof of Proposition \ref{Pr:Constr}, and we claim that it is a group. Indeed, we have $(c,z)\in N_\lambda(Q)=N(Q)$ if and only if $(c,z)=(c,z)R_{(a,x),(b,y)}$ for every $(a,x)$, $(b,y)$, that is, by \eqref{Eq:InnerMaps},
\begin{equation}\label{Eq:ToBeGroup}
    z = (zU_aU_b+y(U_{-c-a}-U_{-c}U_{-a}))U_{a+b}^{-1}
\end{equation}
for every $(a,x)$, $(b,y)$. As $U_{b+a}-U_bU_a = -baA^2=0$ for every $a$, $b$, we see that \eqref{Eq:ToBeGroup} holds.
\end{remark}

\section{Open problems}\seclabel{problems}

\begin{problem}
Are the following two statements equivalent for a finite automorphic loop $Q$?
\begin{enumerate}
\item[(i)] $Q$ has order a power of $2$.
\item[(ii)] Every element of $Q$ has order a power of $2$.
\end{enumerate}
\end{problem}

\begin{problem}
Let $p$ be a prime. Are all automorphic loops of order $p^2$ associative?
\end{problem}

\begin{problem}
Let $p$ be a prime. Is there an automorphic loop of order a power of $p$ and with trivial middle nucleus?
\end{problem}

\begin{problem}
Let $p$ be a prime. Are there automorphic loops of order $p^3$ that are not centrally nilpotent and that are not constructed by Proposition \ref{Pr:Constr}?
\end{problem}

\begin{conjecture}\label{Cj:p}
Let $p$ be a prime and $F=GF(p)$. Let $A$, $B\in GL(2,p)$ be such that $FI\oplus FA$ and $FI\oplus FB$ are anisotropic planes. Then the loops $Q(A)$, $Q(B)$ constructed by \eqref{Eq:Construction} are isomorphic if and only if $A$, $B$ are of the same type.
\end{conjecture}

We have verified Conjecture \ref{Cj:p} computationally for $p\le 5$. Taking advantage of Lemma \ref{Lm:3Types}, we can therefore conclude:

If $p=2$, there is one isomorphism type of loops $Q(A)$ obtained from the matrix
\begin{displaymath}
    \mat{0}{1}{1}{1}
\end{displaymath}
of type $2$---this is the unique commutative automorphic loop of order $8$ that is not centrally nilpotent, constructed already in \cite{JKV2}. If $p=3$, there are two isomorphism types of loops $Q(A)$, corresponding to matrices
\begin{displaymath}
    \mat{0}{1}{2}{0},\quad \mat{1}{1}{2}{1}
\end{displaymath}
of types $1$ and $3$, respectively. If $p=5$, there are three isomorphism types. If Conjecture \ref{Cj:p} is valid for a prime $p>5$, then there are three isomorphism types of loops $Q(A)$ for that prime $p$, according to Lemma \ref{Lm:3Types}.

\section*{Acknowledgement}

After this paper was submitted for publication, P. Cs\"{o}rg\H{o} obtained a stronger result than Theorem \ref{Th:Main} by her signature technique of connected group transversals. Namely, she proved:

\begin{theorem}[Cs\"{o}rg\H{o} \cite{Csorgo}]\label{Th:Csorgo}
If $Q$ is a finite commutative automorphic $p$-loop ($p$ an odd prime), then the multiplication group $\Mlt Q$ is a $p$-group.
\end{theorem}

By a result of Albert \cite{Albert}, $Z(\Mlt Q) \cong Z(Q)$. In particular, if $\Mlt Q$ is a $p$-group then $Z(Q)$ is nontrivial. Our Theorem \ref{Th:Main} then follows from Theorem \ref{Th:Csorgo} by an easy induction on the order of $Q$ (as observed by Cs\"org\H{o} in \cite[Cor. 3.2]{Csorgo}).

Actually, in hindsight it is not difficult to obtain Cs\"{o}rg\H{o}'s Theorem \ref{Th:Csorgo} from our Theorem \ref{Th:Main}: In \cite{Shch} (see also \cite{SaSh}), Shchukin proved that a commutative automorphic loop $Q$ is nilpotent of class at most $n$ if and only if $\Mlt Q$ is nilpotent of class at most $2n-1$. Now suppose $Q$ is a commutative automorphic $p$-loop, $p$ odd. By Theorem \ref{Th:Main}, $Q$ is nilpotent. By the result of Shchukin, $\Mlt Q$ is nilpotent, hence a direct product of groups of prime power order. Since $Z(\Mlt Q) \cong Z(Q)$, it follows that $Z(\Mlt Q)$ is a $p$-group. But then so is $\Mlt Q$.

Finally, we are pleased to acknowledge the assistance of \Prover\ \cite{McCune}, an automated deduction tool,
\Mace\ \cite{McCune}, a finite model builder, and the \textsc{GAP} \cite{GAP} package \textsc{Loops} \cite{LOOPS}.
\Prover\ was indispensable in the proofs of the lemmas leading up to Theorem \thmref{centers}.
We used \Mace\ to find the first automorphic loop of exponent $3$ with trivial center
in \S\secref{noncomm}. We used the \textsc{Loops} package to verify Conjecture \ref{Cj:p} for $p\le 5$.

\end{document}